\documentclass[12pt]{extarticle}
\usepackage{extsizes}
\usepackage[reqno]{amsmath} % equation numbers on the right
\usepackage{amsmath}
\usepackage{amssymb}
\usepackage{amsthm}
\usepackage{amscd}
\usepackage{amsfonts}
\usepackage{graphicx}
\usepackage{dsfont}
\usepackage{fancyhdr}
\usepackage{enumerate}
\usepackage{youngtab}
\usepackage[utf8]{inputenc}
\usepackage{comment}
\usepackage{mathtools}
\usepackage{subfigure}
\usepackage[margin=1.5in]{geometry}
\usepackage{graphicx}
\usepackage{algorithm}
\usepackage{algpseudocode}
\usepackage{color}
\usepackage[hidelinks]{hyperref}
\usepackage{notation}

\newtheorem{corollary}{Corollary}

\newtheorem{lemma}[corollary]{Lemma}
\newtheorem{lemma*}[lem6]{Lemma}

\newtheorem{proposition}[corollary]{Proposition}
\newtheorem{theorem}[corollary]{Theorem}
\newtheorem{example}[corollary]{Example}
\newtheorem{question}[corollary]{Question}
\newtheorem{conjecture}[corollary]{Conjecture}

\newcommand{\defeq}{\vcentcolon=}

\begin{document}

\AtEndDocument{%
  \par
  \medskip
  \begin{tabular}{@{}l@{}}%
    \textsc{Thomás Jung Spier} \\
    \textsc{Dept. of Combinatorics and Optimization} \\ 
    \textsc{University of Waterloo, Canada} \\
    \textit{E-mail address}: \texttt{tjungspier@uwaterloo.ca}
  \end{tabular}}

\title{Efficient reconstruction of the characteristic polynomial}
\author{Thomás Jung Spier
\footnote{tjungspier@uwaterloo.ca}}
\date{\today}
\maketitle
\author
\vspace{-0.8cm}

\begin{abstract} 
    The polynomial reconstruction problem, introduced by Cvetković in 1973, asks whether the characteristic polynomial $\phi^G$ of a graph $G$ with at least $3$ vertices can be reconstructed from the polynomial deck $\{\phi^{G \setminus i}\}_{i \in V(G)}$. In this work, we prove that $\phi^G \pmod{4}$ can be reconstructed from the polynomial deck if the number of vertices in $G$ is even or if the rank of the walk matrix of $G$ over $\Fds_2$ is less than $\lceil n/2 \rceil$. We also prove that for every graph $G$, $\phi^{\overline{G}}\pmod{4}$ can be computed from $\phi^G\pmod{4}$, strengthening a recent result by Ji, Tang, Wang and Zhang. Finally, Hagos showed that the pair of characteristic polynomials $(\phi^G, \phi^{\overline{G}})$ is reconstructible from the generalized polynomial deck $\{(\phi^{G \setminus i}, \phi^{\overline{G} \setminus i})\}_{i \in V(G)}$. We also present an efficient version of this result that requires less information.
\end{abstract}

\begin{center}
\end{center}

%%%%%%%%%%%%%%%%%%%%%%%%%%%%%%%%%%%%%%%%%%%%%%%%%%%%%%%%%%%%%%%%%%%%%%%%%%%%%%%%

\section{Introduction}\label{sec:introduction}

The well-known Ulam-Kelly reconstruction conjecture~\cite{kelly1942isometric} asks whether every graph $G$ with at least three vertices can be uniquely determined from its deck, which is the multiset of its vertex-deleted subgraphs $\{G \setminus i\}_{i \in V(G)}$. As early as 1973, Cvetković (see \cite[p. 69]{cvetkovic1988recent}, \cite{gutman1975reconstruction}) posed the following spectral analogue of Ulam-Kelly's reconstruction conjecture.

To state this problem, let $G$ be a graph with adjacency matrix $A$. The characteristic polynomial of $G$ is defined as $\phi^G(x) \coloneqq \det(xI - A)$. We assume that $G$ has $n$ vertices, with its vertex set indexed by $[n]$. The polynomial deck of $G$ is then defined as the multiset of characteristic polynomials of its vertex-deleted subgraphs $\{\phi^{G \setminus i}\}_{i \in [n]}$.

\begin{question}[Polynomial reconstruction problem~\cite{cvetkovic1988recent}]\label{question:polynomial_reconstruction} Given two graphs $G$ and $H$ with at least $3$ vertices, is it true that if they have the same polynomial deck, then $\phi^G=\phi^H$?
\end{question}

The assumption that the graph has at least three vertices arises from the fact that the two graphs with two vertices have identical polynomial decks, yet their characteristic polynomials differ by $1$. 

The following well-known result is part of the motivation for Question~\ref{sec:questions} and shows that the question is actually equivalent to the reconstruction of the constant coefficient of the characteristic polynomial from the polynomial deck.

\begin{theorem}[Theorem 1.5 in~\cite{GodsilAlgebraicCombinatorics}]\label{thm:charac_derivative} Let $G$ be a graph. Then,
\[\partial_x\phi^G(x) = \displaystyle\sum_{i\in [n]}\phi^{G\setminus i}(x).\]
\end{theorem} 

After 50 years, no counterexample has been found for Question~\ref{question:polynomial_reconstruction}. Schwenk ~\cite{schwenk1979spectral} suspected that this problem had a negative solution, but that counterexamples would be difficult to find. Question~\ref{question:polynomial_reconstruction} is known to have a positive solution for polynomial decks of graphs with at most 10 vertices, regular graphs, trees~\cite{cvetkovic1998seeking}, and unicyclic graphs~\cite{simic2007polynomial}. Additionally, partial progress has been made for polynomial decks of other graph families, such as bipartite~\cite{gutman1975reconstruction} and disconnected graphs~\cite{sciriha2001polynomial} (or see~\cite{farrugia2021graphs}). For a comprehensive overview of the known results on Question~\ref{question:polynomial_reconstruction}, we recommend the article by Sciriha and Stani{\'c}~\cite{sciriha2023polynomial}.

If we seek additional information beyond the polynomial deck that guarantees the reconstruction of the characteristic polynomial, the earliest result dates back to 1979, when Tutte~\cite{tutte1979all} proved that the characteristic polynomial $\phi^G$ is reconstructible from the deck of $G$. To prove this, he showed that the total number of Hamiltonian cycles of $G$ is reconstructible from its deck and used this to obtain $\phi^G$. In the same article, Tutte proved that if $\phi^G$ is irreducible, then $G$ can be reconstructed from its deck.

The following stronger result, which uses less additional information to determine the characteristic polynomial, was obtained in 2000 by Hagos~\cite{hagos2000characteristic}. To state the result, we denote by $\overline{G}$ the complement of the graph $G$, and by $\overline{A}$ its adjacency matrix. Note that if $J$ is the all-ones matrix, then $J = I + A + \overline{A}$. The generalized polynomial deck of $G$ is the multiset of pairs of characteristic polynomials $\{(\phi^{G \setminus i}, \phi^{\overline{G} \setminus i})\}_{i \in [n]}$.

\begin{theorem}[Hagos~\cite{hagos2000characteristic}]\label{thm:hagos} Given two graphs $G$
and $H$, if they have the same generalized polynomial deck, then $(\phi^G, \phi^{\overline{G}})=(\phi^H, \phi^{\overline{H}})$.
\end{theorem}

Contrary to Tutte's result, which can, in principle, lead to the computation of $\phi^G$ (albeit inefficiently, since it involves computing Hamiltonian cycles), it turns out that the proof of Theorem~\ref{thm:hagos} is not constructive. Therefore, the algorithm for obtaining $(\phi^G, \phi^{\overline{G}})$, which is implicit in Theorem~\ref{thm:hagos}, requires the computation of all possible characteristic polynomials and generalized polynomial decks for graphs with $n$ vertices. This algorithm is clearly inefficient, as the number of graphs on $n$ vertices is asymptotically $\dfrac{2^{\binom{n}{2}}}{n!}$ (see~\cite{flajolet2009analytic}).

Our first result addresses this issue by strengthening Theorem~\ref{thm:hagos} and providing an simple and efficient algorithm to reconstruct $(\phi^G, \phi^{\overline{G}})$ from the generalized polynomial deck. It turns out that our reconstruction algorithm requires less information than the generalized polynomial deck.

\begin{theorem}\label{thm:main_theorem} The pair of characteristic polynomials $(\phi^G, \phi^{\overline{G}})$ can be efficiently reconstructed from the generalized polynomial deck. The reconstruction algorithm only requires $\phi^{G \setminus i}$ and the top $\lceil \frac{n + 4}{2} \rceil$ coefficients of $\phi^{\overline{G} \setminus i}$ for every $i\in[n]$.
\end{theorem}

Although we do not explicitly determine the time complexity of our algorithm, it is clearly polynomial, as it only involves elementary operations with rational power series.

It turns out that for almost all graphs, our reconstruction algorithm can be made to require even less information. Recall that the walk matrix $W^G$ of a graph $G$ is defined by
\[
W^G\defeq(\mathbf{1}\quad A\mathbf{1}\quad \cdots\quad A^{n-1}\mathbf{1}),
\]
\noindent where $\mathbf{1}$ is the all $1$'s vector.

\begin{theorem}\label{thm:main_theorem_2} Assume that we know the top $ \lceil\frac{2n + 4}{3}\rceil $ coefficients of the pairs of polynomials $ (\phi^{G \setminus i}, \phi^{\overline{G} \setminus i}) $ for every $i$ in $[n]$. Then, it is possible to efficiently determine whether the rank of the walk matrix $W^G$ is at least $\lfloor\frac{n-1}{3}\rfloor$, and if so,  efficiently reconstruct $(\phi^G, \phi^{\overline{G}})$.
\end{theorem}

It was proved by O'Rourke and Touri~\cite{o2016conjecture}, confirming a conjecture by Godsil~\cite{godsil2012controllable}, that almost all graphs are controllable, meaning that the walk matrix $W^G$ has full rank. As a consequence, Theorem~\ref{thm:main_theorem_2} applies to almost all graphs. 

Theorems~\ref{thm:main_theorem} and~\ref{thm:main_theorem_2} also have implications for the original Ulam-Kelly reconstruction conjecture. It was proved by Godsil and McKay~\cite{GodsilMcKay}, generalizing the earlier result by Tutte~\cite{tutte1979all}, that if $G$ is controllable, then it is efficiently reconstructible from its walk matrix $W^G$ and the pair of characteristic polynomials $(\phi^G, \phi^{\overline{G}})$. It is easy to show (see Section~\ref{sec:proof_main_theorem}) that $W^G$ can be efficiently computed from $(\phi^G, \phi^{\overline{G}})$ and the generalized polynomial deck of $G$.

By combining these observations with Theorem~\ref{thm:main_theorem_2} and the fact that almost all graphs are controllable, we obtain the following result.

\begin{corollary}\label{cor:main_corollary} Almost all graphs $G$ are efficiently reconstructible from the top $\lceil \frac{2n + 4}{3} \rceil$ coefficients of the pairs of polynomials $(\phi^{G \setminus i}, \phi^{\overline{G} \setminus i})$ for every $i$ in $[n]$. In particular, almost all graphs are efficiently reconstructible from the generalized polynomial deck.
\end{corollary}

A result of Johnson and Newman~\cite{johnson1980note} (see also~\cite[p. 6]{godsil2012controllable}) states that two graphs $G$ and $H$ satisfy $(\phi^G, \phi^{\overline{G}}) = (\phi^H, \phi^{\overline{H}})$ if and only if there exists an orthogonal matrix $Q$ such that $Q^T A_G Q = A_H$ and $Q \mathbf{1} = \mathbf{1}$. It also follows from Theorem~\ref{thm:main_theorem} that the existence of this orthogonal matrix is efficiently decidable from the generalized polynomial decks of $G$ and $H$.

Using the same set of techniques as in the proofs of Theorems~\ref{thm:main_theorem} and~\ref{thm:main_theorem_2}, we also prove the following result, identifying a new infinite family of graphs (albeit a vanishingly small fraction of all graphs) for which Question~\ref{question:polynomial_reconstruction} has a positive solution.

\begin{theorem}\label{thm:main_theorem_3} From the polynomial deck, it is possible to determine whether $G$ has no cycles of length $4$, and if so, whether the rank of the walk matrix is at most $2$. If both conditions are satisfied, then $(\phi^G, \phi^{\overline{G}})$ can be efficiently reconstructed.
\end{theorem}

Note that regular graphs are precisely those for which the walk matrix has rank $1$, and it is straightforward to show that Question~\ref{question:polynomial_reconstruction} has a positive solution in this case. Graphs whose walk matrix has rank $2$ are referred to in the literature as $2$-main graphs (see~\cite{hayat2016note}).

Another approach to Question~\ref{question:polynomial_reconstruction} focuses on the properties of a graph that can be deduced from its polynomial deck. It is known that the degree sequence, the length of the shortest odd cycle, and the number of triangles, quadrangles, and pentagons can be determined from the polynomial deck. A summary of the known results in this direction is provided in Lemmas 4.2 and 4.3 of~\cite{sciriha2023polynomial}. 

Our next results provide additional information that can be inferred from the polynomial deck. 

\begin{theorem}\label{thm:main_theorem_4} Let $G$ be a graph with at least $3$ vertices. Then, the constant coefficients of $\phi^G\pmod{2}$ and $\phi^{\overline{G}}\pmod{2}$, as well as $W^G\pmod{2}$ and the top $n$ coefficients of $\phi^{\overline{G}}\pmod{4}$, can be reconstructed from the polynomial deck of $G$.
\end{theorem}

We remark that if the number of vertices $n$ is odd, the constant coefficient of $\phi^G \pmod{2}$ is always even (Lemma~\ref{lem:cte_mod2_even}). Therefore, in the previous result, the ability to determine the constant coefficient of $\phi^G \pmod{2}$ is non-trivial only when $n$ is even. However, the information we obtain about $W^G \pmod{2}$ and $\phi^{\overline{G}} \pmod{4}$ in the previous result is non-trivial and was previously unknown for all values of $n$.

If the number of vertices $n$ is even, or if an additional condition on the rank of $W^G \pmod{2}$ is satisfied, we can, in fact, obtain a stronger result.

\begin{theorem}\label{thm:main_theorem_5} Let $G$ be a graph with $n\geq 3$ vertices. If $n$ is even or if the rank of $W^G$ over $\Fds_2$ is less than $\lceil \frac{n}{2}\rceil$, then, the constant coefficient of $\phi^G\pmod{4}$, as well as $\phi^{\overline{G}}\pmod{4}$, can be  reconstructed from the polynomial deck of $G$.
\end{theorem}

Wang~\cite{wang2017simple} (see Lemma~\ref{lem:wang_walk_matrix2}) proved that the rank of $W^G$ over $\mathbb{F}_2$ is at most $\lceil \frac{n}{2} \rceil$. Therefore, the conclusion of Theorem~\ref{thm:main_theorem_5} remains unknown only in the case where $n$ is odd and the rank of $W^G$ over $\mathbb{F}_2$ is exactly $\lceil \frac{n}{2} \rceil$. Determining whether Theorem~\ref{thm:main_theorem_5} holds in this special case remains an interesting open problem.

As a consequence of Theorems~\ref{thm:main_theorem_4} and~\ref{thm:main_theorem_5}, we obtain the following result.

\begin{corollary}\label{cor:main_corollary2} If $G$ and $H$ are graphs on $n$ vertices that form a counterexample pair to the polynomial reconstruction problem, then $\phi^G = \phi^H + 2k$ for some $k$ in $\Zds\setminus \{0\}$. Moreover, if $n$ is even, then $k$ is also even.
\end{corollary}

To the best of the author's knowledge, Corollary~\ref{cor:main_corollary2} is the only result that provides non-trivial information about the constant coefficients of $\phi^G$ and $\phi^H$ for any potential counterexample pair $(G, H)$ to the polynomial reconstruction problem (Question~\ref{question:polynomial_reconstruction}).

Although our proofs of Theorems~\ref{thm:main_theorem_4} and~\ref{thm:main_theorem_5} do not address signed graphs (though we believe the results still apply), signed cycles (balanced and unbalanced) share the same polynomial deck but have characteristic polynomials that differ by a multiple of four (see~\cite[Section 12.4]{sciriha2023polynomial}). As a consequence, at least for signed graphs, there exist counterexamples to the analogue of Question~\ref{question:polynomial_reconstruction} that satisfy the conclusion of Corollary~\ref{cor:main_corollary2}.

A key component of our approach to Theorem~\ref{thm:main_theorem_4} is the following recent result by Ji, Thang, Wang, and Zhang~\cite{ji2024new}, obtained in their study of graphs determined by the generalized spectrum.

\begin{theorem}[Theorem 1.6 in~\cite{ji2024new}]\label{thm:wang_et_al} Let $G$ and $H$ be graphs with $\phi^G = \phi^H$. Then, $\phi^{\overline{G}} \equiv \phi^{\overline{H}}\pmod{4}$.
\end{theorem}

If we are given the polynomial deck $\{\phi^{G\setminus i}\}_{i \in [n]}$, Theorem~\ref{thm:wang_et_al} allows us to immediately obtain $\{(\phi^{G\setminus i}, \phi^{\overline{G}\setminus i} \pmod{4})\}_{i \in [n]}$, which is a version of the generalized polynomial deck modulo $4$. This suggests that the techniques used in the proof of Theorem~\ref{thm:main_theorem} can also be directly applied to Theorems~\ref{thm:main_theorem_4} and~\ref{thm:main_theorem_5}. Unfortunately, this is true only to a certain extent, as several significant technical issues arise from the use of arithmetic modulo $4$.

In our investigations leading to Theorems~\ref{thm:main_theorem_4} and~\ref{thm:main_theorem_5}, we also found that Theorem~\ref{thm:wang_et_al} can be strengthened to the following result.

\begin{theorem}\label{thm:main_theorem_6} For any graph $G$, given $\phi^G\pmod{4}$ we can compute $\phi^{\overline{G}}\pmod{4}$.
\end{theorem}

Computational experiments suggest that, in general, $\phi^{\overline{G}}\pmod{p^k}$ is not determined by $\phi^G\pmod{p^k}$ or $\phi^G$ for any choice of prime power $p^k$ other than $2$ or $2^2$.

As explained in Section~\ref{sec:charac_Gc}, Theorem~\ref{thm:main_theorem_6} follows directly from Proposition 3.4 in~\cite{ji2024new} and a recent result by Anni, Ghitza, and Medvedovsky~\cite{anni2024elementary}. In Section~\ref{sec:charac_Gc}, we also present a more concise proof of this recent result of the last article in our specific context.

In contrast to the other results we have presented, although our proofs of Theorems~\ref{thm:main_theorem_4},~\ref{thm:main_theorem_5} and~\ref{thm:main_theorem_6} are constructive, we do not claim that they lead to efficient algorithms. Finally, we note that while Theorem~\ref{thm:main_theorem_6} holds, the conclusion of Theorem~\ref{thm:main_theorem_5} does not hold if we only have access to the polynomial deck modulo $4$, i.e., $\{\phi^{G \setminus i} \pmod{4}\}_{i \in [n]}$, rather than the full polynomial deck, as demonstrated by the following example.

\begin{example} Let $n \equiv 2 \pmod{4}$, with $n \geq 6$. Consider the star graph $S_n$ with $n$ vertices, where one vertex has degree $n-1$ and the remaining $n-1$ vertices have degree $1$. Then, $\phi^{S_n}(x) = x^n - (n-1)x^{n-2} \equiv x^n - x^{n-2} \pmod{4}$, and the polynomial deck modulo $4$ of $S_n$ is $\{x^{n-1} \pmod{4}\}_{i\in [n]}$. On the other hand, if $D_n$ is the graph formed by $n$ vertices with no edges, then $\phi^{D_n}(x) = x^n$, and the polynomial deck modulo $4$ is also $\{x^{n-1} \pmod{4}\}_{i\in [n]}$.
\end{example}

In Section~\ref{sec:proof_main_theorem}, we prove Theorems~\ref{thm:main_theorem} and~\ref{thm:main_theorem_2} as corollaries of a more general statement and also prove Theorem~\ref{thm:main_theorem_3}. In Section~\ref{sec:charac_Gc}, we prove Theorem~\ref{thm:main_theorem_6}, and in Section~\ref{sec:reconstruction_mod4}, we prove Theorems~\ref{thm:main_theorem_4} and~\ref{thm:main_theorem_5}. Finally, in Section~\ref{sec:questions}, we provide motivation for this work, originating from homomorphism counts, and present some related questions.

%%%%%%%%%%%%%%%%%%%%%%%%%%%%%%%%%%%%%%%%%%%%%%%%%%%%%%%%%%%%%%%%%%%%%%%%%%%%%%%%

\section{Efficient reconstruction}\label{sec:proof_main_theorem}\

In this section, we prove Theorems~\ref{thm:main_theorem} and~\ref{thm:main_theorem_2} as corollaries of a more general result. Our proof of this result follows some ideas introduced by Hagos~\cite{hagos2000characteristic} and uses walk generating function identities due to Godsil and McKay~\cite{GodsilMcKay}. We consider the following walk generating functions:

\begin{equation}\label{eq:closed_walk}
w^G_{i, i}(x)\defeq\displaystyle\sum_{k\geq 0}\dfrac{e_i^TA^ke_i}{x^{k+1}} = \dfrac{\phi^{G\setminus i}}{\phi^G}(x),
\end{equation}
\begin{equation}\label{eq:all_closed_walks}
\displaystyle\sum_{i\in [n]}w^G_{i,i}(x)=\displaystyle\sum_{k\geq 0}\dfrac{\tr{A^k}}{x^{k+1}}=\dfrac{\partial_x\phi^G}{\phi^G}(x),
\end{equation}
\begin{equation}\label{eq:walk}
w^G_i(x)\defeq\displaystyle\sum_{k\geq 0}\dfrac{e_i^TA^k\mathbf{1}}{x^{k+1}},
\end{equation}
\begin{equation}\label{eq:all_walks}
w^G(x)\defeq\displaystyle\sum_{i\in [n]} w^G_i(x)=\displaystyle\sum_{k\geq 0}\dfrac{\mathbf{1}^TA^k\mathbf{1}}{x^{k+1}}=\dfrac{(-1)^n\phi^{\overline{G}}(-x-1)-\phi^G(x)}{\phi^G(x)}.
\end{equation}

The generating functions in Equations~\eqref{eq:closed_walk} and~\eqref{eq:walk} count the total number of closed walks and walks starting at a given vertex $i$ in the graph $G$, respectively, while the generating functions in Equations~\eqref{eq:all_closed_walks} and~\eqref{eq:all_walks} count the total number of closed walks and walks in the graph $G$, respectively. Note that the coefficients of $w^G_i(x)$ correspond to the $i$-th row of the walk matrix $W^G$. Godsil and McKay~\cite{GodsilMcKay} (or see~\cite{hagos2000characteristic}) obtained a formula that connects some of these walk generating functions.

\begin{theorem}[Equation 2 in~\cite{hagos2000characteristic}]\label{thm:godsil_mckay} Let $G$ be a graph and $i$ one of its vertices. Then,
\[
w^G(x) = w^{G \setminus i}(x) + \frac{w^G_i(x)^2}{w^G_{i,i}(x)}.
\]
\end{theorem}

We now present some preliminary results that will be needed for the proof of our main result. The following results provide elementary properties of generating functions that will be useful.

\begin{proposition}\label{prop:gen_basics} Consider $S(x) = \sum_{k \geq 0} \frac{s_k}{x^{k+1}}$ and $T(x) = \sum_{k \geq 0} \frac{t_k}{x^{k+1}}$ with $s_k, t_k \in \Cds$ for all $k \geq 0$. If we know the first $m$ coefficients of $S$ and $T$, then we can  compute the first $m$ and $m+1$ coefficients of $S \pm T$ and $S \cdot T$, respectively. Furthermore, if we know $s_0\neq 0$ and the first $m$ coefficients of $S^2$, then we can  compute the first $m-1$ coefficients of $S$.
\end{proposition}
\begin{proof} If we know the first $m$ coefficients of $S$ and $T$, then clearly, we can obtain the first $m$ coefficients of $S \pm T$. Note that the first coefficient of $S \cdot T$ is 0, and for $k \geq 2$, the $k$-th coefficient is given by $\sum_{i=0}^{k-2} s_i t_{k-2-i}$. Therefore, the first $m+1$ coefficients of $S \cdot T$ are also determined.

Let $S(x)^2 = \sum_{k \geq 0} \frac{u_k}{x^{k+1}}$ with $u_k \in \Cds$ for every $k \geq 0$, and assume we know its first $m$ coefficients. Observe that $u_0 = 0$ and for $k \geq 1$, we have $u_k = \sum_{i=0}^{k-1} s_i s_{k-1-i}$, which can be rewritten as $2s_0 s_{k-1} = u_k - \sum_{i=1}^{k-2} s_i s_{k-1-i}$. Thus, if we know $u_k$ and $s_0, \dots, s_{k-2}$, we can compute $2s_0 s_{k-1}$. Since we know $s_0 \neq 0$, we can then determine $s_{k-1}$. By iterating this procedure, we can recursively compute the first $m$ coefficients of $S$.
\end{proof}

\begin{proposition}\label{prop:gen_basics_2} Let $p(x)=a_0 x^{n-1}+a_1 x^{n-2}+\cdots +a_{n-1}$ and $q(x)=x^n+b_1 x^{n-1}+\cdots +b_n$ be polynomials in $\Cds[x]$. Then, $\frac{p}{q}(x)=\sum_{k\geq 0}\frac{s_k}{x^{k+1}}$ with
\[\label{eq:gen_1}
\left[ {\begin{array}{ccccc}
	s_0  & 0 & \cdots & \cdots  & 0 \\
    s_1 & \ddots & \ddots & \ddots  & \vdots \\
    \vdots & \ddots & \ddots & \ddots  & 0 \\ 
    s_{n-1} & \cdots & s_1 & s_0 & 0 \\
	\end{array} } \right]\cdot \left[ {\begin{array}{c}
	1 \\
    b_1 \\
    b_2 \\ 
    \vdots \\ 
    b_n \\
	\end{array} } \right] = \left[ {\begin{array}{c}
	a_0 \\
    a_1 \\
    \vdots \\ 
    a_{n-1} \\
	\end{array} } \right],
\]
\noindent and,
\[
\left[ {\begin{array}{ccccc}
	s_n & \cdots & \cdots & s_1 & s_0 \\
    \vdots & \ddots & \ddots & \ddots  & s_1 \\
    \vdots & \ddots & \ddots & \ddots  & \vdots \\
    s_{2n-1} & \cdots & \cdots & s_n & s_{n-1} \\
	\end{array} } \right]\cdot \left[ {\begin{array}{c}
	1 \\
    b_1 \\
    b_2 \\ 
    \vdots \\ 
    b_n \\
	\end{array} } \right] = \left[ {\begin{array}{c}
	0 \\
    \vdots \\
    \vdots \\
    0 \\
	\end{array} } \right].
\]
In particular, if we know the top $m$ coefficients of $p$ and $q$, then we can  compute $s_0,\dots, s_{m-1}$.
\end{proposition}
\begin{proof} This immediately follows by writing $p(x)=q(x)\sum_{k\geq 0}\frac{s_k}{x^{k+1}}$ and comparing coefficients on both sides.
\end{proof}

\begin{proposition}\label{prop:gen_basics_3} Let $S(x)=\sum_{k\geq 0} \frac{s_k}{x^{k+1}}$ with $s_k$ in $\Cds$ for every $k\geq 0$. If $S(x)$ is rational, then it can be written as $(p/q)(x)$ with $p,q\in\Cds[x]$, $q$ monic, $\gcd(p, q)=1$ and $m\defeq\deg(q)\geq\deg(p)+1$. In this case,
\[
\left[ {\begin{array}{ccccc}
	s_{l-1} & \cdots & s_1 & s_0 \\
    \vdots & \ddots & \ddots  & s_1 \\
    \vdots & \ddots & \ddots  & \vdots \\
    s_{2l-2} & \cdots & \cdots & s_{l-1} \\
	\end{array} } \right]
\]
\noindent is invertible if, and only if, $l\leq m$.
\end{proposition}
\begin{proof} This is a consequence of Proposition~\ref{prop:gen_basics_2}. Observe that $m$ is the length of the smallest linear recurrence satisfied by $(s_k)_{k\geq 0}$. Consider the matrix in the statement for some $l$. This matrix is not invertible if and only if the columns of the matrix are linearly dependent, which corresponds to a linear recurrence of length at most $l-1$ for $(s_k)_{k\geq 0}$. 
\end{proof}

\begin{proposition}\label{prop:gen_basics_4} Let $w^G(x)=\frac{p}{q}(x)$, where $p, q\in\Cds[x]$, $q$ is monic, and $\gcd(p,q)=1$. Then, the degree of $q$ is equal to the rank of the walk matrix $W^G$.
\end{proposition}
\begin{proof} This is true by~\cite[p. 3]{godsil2012controllable} and Equation~\eqref{eq:all_walks}.
\end{proof}

Now we are ready for the proof of the main result of this section.

\begin{theorem}\label{thm:main_general} Let $s \leq t \leq n$. Assume that we know the top $s$ and $t$ coefficients of $\phi^{\overline{G} \setminus i}$ and $\phi^{G \setminus i}$, respectively, for every $i$ in $[n]$. If $s + \frac{t}{2} \geq n + 2$, then it is possible to  determine whether the rank of the walk matrix $W^G$ is at least $n + 1 - t$, and if so, to  compute $(\phi^G, \phi^{\overline{G}})$.
\end{theorem}
\begin{proof} First, note that by Theorem~\ref{thm:charac_derivative}, we know the top $s$ and $t$ coefficients of $\phi^{\overline{G}}$ and $\phi^G$, respectively. As a consequence, we can determine the top $s-1$ coefficients of $(-1)^n \phi^{\overline{G}}(-x-1) - \phi^G(x)$ and $(-1)^{n-1} \phi^{\overline{G} \setminus i}(-x-1) - \phi^{G \setminus i}(x)$ for every $i$ in $[n]$. By Equation~\eqref{eq:all_walks} and Proposition~\ref{prop:gen_basics_2}, we then know the first $s-1$ coefficients of $w^G(x)$ and $w^{G \setminus i}(x)$ for every $i$ in $[n]$. Finally, by Equation~\eqref{eq:closed_walk} and Proposition~\ref{prop:gen_basics_2}, we can determine the first $s$ coefficients of $w^G_{i,i}(x)$ for every $i$ in $[n]$.

By Theorem~\ref{thm:godsil_mckay}, we have the equation
\[
w^G_i(x)^2 =(w^G(x)-w^{G \setminus i}(x))w^G_{i,i}(x),
\]
from which, by Proposition~\ref{prop:gen_basics}, it follows that we know the first $s$ coefficients of $w^G_i(x)^2$ for every $i$ in $[n]$. Since the first coefficient of $w^G_i(x)$ is 1, by Proposition~\ref{prop:gen_basics}, we can  compute the first $s-1$ coefficients of $w^G_i(x)$ for every $i$ in $[n]$.

But then, by Equation~\eqref{eq:walk}, and since
\begin{align}\label{eq:main_mechanism}
\displaystyle\sum_{i\in [n]} e_i^TA^k\mathbf{1}\cdot e_i^TA^l\mathbf{1} &= \displaystyle\sum_{i\in [n]} \mathbf{1}^TA^k e_ie_i^TA^l\mathbf{1} \nonumber\\ &=\mathbf{1}^TA^k\left(\displaystyle\sum_{i\in[n]} e_ie_i^T\right)A^l\mathbf{1} \nonumber\\ &= \mathbf{1}^TA^{k+l}\mathbf{1},
\end{align}
\noindent for every $k$ and $l$ in $\Zds_{\geq 0}$, we obtain the first $2s-3$ coefficients of $w^G(x)$.

Write $w^G(x)=\sum_{k\geq 0}\frac{w_k}{x^{k+1}}$ and $\phi^G(x)=x^n+b_1 x^{n-1}+\cdots +b_n$. Then, we know $w_k$ for $k\in \{0,1\dots, 2s-4\}$ and $b_j$ for $j\in\{1,\dots, t-1\}$. Using the second equation in Proposition~\ref{prop:gen_basics_2}, and observing that $s+\frac{t}{2}\geq n+2$ is equivalent to $2s-4\geq 2n-t$, we obtain
\begin{equation}\label{eq:determine_coeff}
\left[ {\begin{array}{ccccc}
	w_n & \cdots & \cdots & w_1 & w_0 \\
    \vdots & \ddots & \ddots & \ddots  & w_1 \\
    \vdots & \ddots & \ddots & \ddots  & \vdots \\
    w_{2n-t} & \cdots & \cdots & \cdots & w_{n-t} \\
	\end{array} } \right]\cdot \left[ {\begin{array}{c}
	1 \\
    b_1 \\
    b_2 \\ 
    \vdots \\ 
    b_n \\
	\end{array} } \right] = \left[ {\begin{array}{c}
	0 \\
    \vdots \\
    \vdots \\
    0 \\
	\end{array} } \right],
\end{equation}
\noindent where $w_0,w_1,\dots,w_{2n-t}$ are known. By Propositions~\ref{prop:gen_basics_3} and~\ref{prop:gen_basics_4} the rank of $W^G$ is at least $n+1-t$, if and only if, the following matrix is invertible
\[
\left[ {\begin{array}{ccccc}
	w_{n-t} & \cdots & w_1 & w_0 \\
    \vdots & \ddots & \ddots  & w_1 \\
    \vdots & \ddots & \ddots  & \vdots \\
    w_{2n-2t} & \cdots & \cdots & w_{n-t} \\
	\end{array} } \right].
\]

If the rank of $W^G$ is at least $n + 1 - t$, that is, if the matrix above is invertible, then, by Equation~\eqref{eq:determine_coeff}, we can determine the coefficients $b_j$ for $j \in \{t, \dots, n\}$. Therefore, we can completely recover the characteristic polynomial $\phi^G$. Using Equation~\eqref{eq:all_walks} and the first equation in Proposition~\ref{prop:gen_basics_2}, we can then recover $\phi^{\overline{G}}$. Thus, we conclude that we can  compute the pair of characteristic polynomials $(\phi^G, \phi^{\overline{G}})$.
\end{proof}

Theorems~\ref{thm:main_theorem} and~\ref{thm:main_theorem_2} immediately follow as corollaries of the previous result. We now proceed with the proof of Theorem~\ref{thm:main_theorem_3}.

\begin{proof}[Proof of Theorem~\ref{thm:main_theorem_3}] First, note that by Theorem~\ref{thm:charac_derivative}, we know the top $n$ coefficients of $\phi^G$. From Equations~\eqref{eq:closed_walk} and Proposition~\ref{prop:gen_basics_2}, we can compute $e_i^T A^k e_i$ for every $i$ in $[n]$ and $k$ in $\{0, 1, \dots, n-1\}$. By Lemma 4.2 $(vii)$ of~\cite{sciriha2023polynomial}, we can also determine whether $G$ has cycles of length $4$.

Assume that $G$ has no cycles of length $4$. Then, for every vertex $i$, there is a bijection between walks of length $2$ starting at $i$ and closed walks of length $4$ starting at $i$. Therefore, we have $e_i^T A^2 \mathbf{1} = e_i^T A^4 e_i$, and we can compute $e_i^T A^k \mathbf{1}$ for every $i$ in $[n]$ and $k$ in $\{0, 1, 2\}$. In other words, we can compute the first three columns of the walk matrix $W^G$.

With the first three columns of $W^G$, it is possible to determine whether the rank of $W^G$ is at most $2$. Assume that the rank of $W^G$ is $2$ (the rank $1$ case is analogous and corresponds to $G$ being regular). In this case, we can obtain a non-trivial linear relation
\[
b_0\mathbf{1}+b_1 A\mathbf{1}+A^2\mathbf{1}=\mathbf{0},
\]
where $b_0,b_1\in \Rds$. As a consequence, we have the recurrence
\[
A^k\mathbf{1}=-b_1 A^{k-1}\mathbf{1} -b_0 A^{k-2}\mathbf{1},
\]
for every integer $k\geq 2$. Therefore, we can determine $A^k \mathbf{1}$ for all $k \geq 0$, and thus obtain all the coefficients of $w^G_i(x)$ for every $i$ in $[n]$. Since $w^G(x) = \sum_{i \in [n]} w^G_i(x)$, we can then obtain all the coefficients of $w^G(x)$.

By Equation~\eqref{eq:all_walks} and the second equation in Proposition~\ref{prop:gen_basics_2}, we can then determine $\phi^G$. Finally, by Equation~\eqref{eq:all_walks} and the first equation in Proposition~\ref{prop:gen_basics_2}, we compute $\phi^{\overline{G}}$.
\end{proof}

Finally, we note that the construction in Theorem 4.1 of~\cite{hayat2016note} shows that, given a connected $2$-main graph that is neither a tree, unicyclic, nor regular, and contains at most one $4$-cycle, it is possible to construct another connected $2$-main graph with twice the number of vertices, no $4$-cycles, and that does not belong to any of these graph classes. By repeatedly applying this construction, starting from a known graph with the desired properties (for example, one of the graphs in~\cite{hu2009bicyclic}), we obtain that Theorem~\ref{thm:main_theorem_3} holds for a new infinite family of graphs.

%%%%%%%%%%%%%%%%%%%%%%%%%%%%%%%%%%%%%%%%%%%%%%%%%%%%%%%%%%%%%%%%%%%%%%%%%%%%%%%%

\section{Characteristic polynomial modulo $4$}\label{sec:charac_Gc}

In this section, we prove Theorem~\ref{thm:main_theorem_6} and present some results that will be required in Section~\ref{sec:reconstruction_mod4}.

The following result established by Ji, Thang, Wang, and Zhang~\cite{ji2024new} is stronger than Theorem~\ref{thm:wang_et_al} in that the top coefficients of $\phi^G$ explicitly determine the top coefficients of $\phi^{\overline{G}}\pmod{4}$.

\begin{theorem}[Proposition 3.4 in~\cite{ji2024new}]\label{thm:wang_series} Let $A$ be the adjacency matrix of a graph $G$ and $m\defeq 2^t(2s+1)$ be an integer with $t$ and $s$ in $\Zds_{\geq 0}$. Then,
\[
\mathbf{1}^T A^{m}\mathbf{1} \equiv \dfrac{\tr{A^{2m}}}{2^{t+1}} + \displaystyle\sum_{l=0}^{t+1} \dfrac{\tr{A^{2^l(2s+1)}}}{2^l}\pmod{4}.
\]
\end{theorem}

To understand how Theorem~\ref{thm:wang_series} implies Theorem~\ref{thm:wang_et_al}, consider that, given $\phi^G$, we can use Equation~\eqref{eq:all_closed_walks} to compute $\tr{A^k}$ for all $k\geq 0$. Then, by Theorem~\ref{thm:wang_series}, we can determine $\mathbf{1}^T A^k\mathbf{1}\pmod{4}$ for every $k\geq0$. From Equation~\eqref{eq:all_walks} and the first equation in Proposition~\ref{prop:gen_basics_2}, we can then obtain $(-1)^n \phi^{\overline{G}}(-x-1)-\phi^G(x)\pmod{4}$, which allows us to compute $\phi^{\overline{G}}\pmod{4}$.

Note that if we could compute $\tr{A^{2^t(2s+1)}} \pmod{2^{t+2}}$ for every $t$ and $s$ in $\Zds_{\geq 0}$, then this method would be sufficient to determine $\phi^{\overline{G}} \pmod{4}$. It turns out that this information is already determined by $\phi^G \pmod{4}$, as the next result shows. This result was obtained independently by the author but proved more generally by Anni, Ghitza, and Medvedovsky~\cite{anni2024elementary}.

For a prime number $p$ and integer $m$, let $v_p(m)$ denote the largest natural $k$ such that $p^k$ divides $m$. Let $v_p(0) \defeq \infty$.

\begin{theorem}[\cite{anni2024elementary}, Prop. 12 and 13]\label{thm:deep_power_sum} Let $p^l$ be a prime power and $k\geq 0$. Then, the top $k+1$ coefficients of $\phi^G \pmod{p^l}$ can be computed from $\tr{A^m} \pmod{p^{v_p(m)+l}}$ for all $m$ in $\{0, 1, \dots, k\}$, and conversely, the values of $\tr{A^m} \pmod{p^{v_p(m)+l}}$ for these $m$ can be computed from the top $k+1$ coefficients of $\phi^G \pmod{p^l}$.  
\end{theorem}

As indicated above, Theorem~\ref{thm:main_theorem_6} follows as a direct consequence of Theorems~\ref{thm:wang_series} and~\ref{thm:deep_power_sum}, applied with the prime power $2^2$. In what follows, we present a proof of Theorem~\ref{thm:deep_power_sum} as a consequence of some other results of~\cite{anni2024elementary}, which will also be needed in Section~\ref{sec:reconstruction_mod4}.

In what follows in this section, we write $\phi^G(x)=x^n-b_1x^{n-1}+b_2x^{n-2}-b_3x^{n-3}+\cdots +(-1)^n b_n$. Let $\theta_1, \dots, \theta_n$ be the $n$ zeros of $\phi^G(x)$, or equivalently, the eigenvalues of the adjacency matrix $A$. Note that if we define $b_k \defeq 0$ for $k > n$, then for every $m \geq 1$, $b_m$ corresponds to the $m$-th elementary symmetric polynomial in $\theta_1, \dots, \theta_n$. On the other hand, $\tr{A^m}$ is the $m$-th power sum symmetric polynomial in $\theta_1, \dots, \theta_n$ for every $m \geq 1$. This correspondence enables us to translate the results from~\cite{anni2024elementary} into our context.

To state the next result, let $\lambda$ be a partition of an integer $m > 0$, denoted $\lambda \vdash m$. A partition $\lambda$ of $m$ is a finite ordered tuple $(\lambda_1, \lambda_2, \dots, \lambda_{k(\lambda)})$ such that $\lambda_1 \geq \lambda_2 \geq \cdots \geq \lambda_{k(\lambda)} \geq 1$ and $\sum_{j=1}^{k(\lambda)} \lambda_j = m$. Here, $k(\lambda)$ denotes the number of parts of the partition $\lambda$. We define $r_j(\lambda)$ as the number of parts of size $j$ in the partition $\lambda$, that is, $r_j(\lambda) = |\{l \in [k(\lambda)] \mid \lambda_l = j\}|$. Note that $\sum_{j \geq 1} r_j(\lambda) = k(\lambda)$, and this sum is finite.

\begin{theorem}[\cite{anni2024elementary}, Cor. 22]\label{thm:deep_power_sum_identity}  
For every $m \geq 1$,  
\[
\tr{A^m} = (-1)^m m \sum_{\lambda \vdash m} \frac{(-1)^{k(\lambda)}}{k(\lambda)} \binom{k(\lambda)}{r_1(\lambda), r_2(\lambda), \dots} \prod_{j \geq 1} b_j^{r_j(\lambda)}.
\]
\end{theorem}

We also need the following results.

\begin{proposition}[\cite{anni2024elementary}, Cor. 24]\label{prop:kummer} Let $p$ be a prime number, and let $\lambda \vdash m$. Then, for every $j \geq 1$ with $r_j(\lambda)\geq 1$,  
\[
v_p\left(\binom{k(\lambda)}{r_1(\lambda), r_2(\lambda), \dots}\right) \geq v_p(k(\lambda)) - v_p(r_j(\lambda)).
\]
\end{proposition}

The next result is well known, and part of its statement can be found in Lemma 25 of~\cite{anni2024elementary}.

\begin{lemma}\label{lem:prime_power} Let $p$ be a prime number. Then, both $x^{p^m} \pmod{p^{m+l}}$ and $p^m x \pmod{p^{m+l}}$ depend only on $x \pmod{p^l}$.
\end{lemma}

\begin{corollary}\label{cor:prime_power_we_know} Let $p$ be a prime number, and let $\lambda \vdash m$. If we know $b_j\pmod{p^l}$ for every $j\leq m$, then we can determine
\[
\frac{m}{k(\lambda)} \binom{k(\lambda)}{r_1(\lambda), r_2(\lambda), \dots} \prod_{j \geq 1} b_j^{r_j(\lambda)}\pmod{p^{v_p(m)+l}}.
\]
\end{corollary}
\begin{proof} Let $v_p(\lambda)$ be defined as the minimum of $\{v_p(r_j(\lambda)) \mid j\geq 1, r_j(\lambda)\geq 1\}$. By Lemma~\ref{lem:prime_power}, as we know $b_j\pmod{p^l}$ for every $j\leq m$, we can determine $\prod_{j \geq 1} b_j^{r_j(\lambda)}\pmod{p^{v_p(\lambda)+l}}$. By Proposition~\ref{prop:kummer}, we have
\[
v_p\left(\frac{m}{k(\lambda)}\binom{k(\lambda)}{r_1(\lambda), r_2(\lambda), \dots}\right) \geq v_p(m) - v_p(r_j(\lambda)),
\]
for every $j \geq 1$ with $r_j(\lambda) \geq 1$. Therefore, the same inequality holds with the right-hand side as $v_p(m) - v_p(\lambda)$. The result then follows from Lemma~\ref{lem:prime_power}.
\end{proof}

Using Theorem~\ref{thm:deep_power_sum_identity} and Corollary~\ref{cor:prime_power_we_know}, we can provide a more concise proof of Theorem~\ref{thm:deep_power_sum} in our context, compared to the approach in~\cite{anni2024elementary}.

\begin{proof}[Proof of Theorem~\ref{thm:deep_power_sum}]
Assume that we know the top $k + 1$ coefficients $1$, $-b_1$, $\dots$, $(-1)^k b_k$ of $\phi^G\pmod{p^l}$. By a direct application of Theorem~\ref{thm:deep_power_sum_identity} and Corollary~\ref{cor:prime_power_we_know}, we can compute $\tr{A^m} \pmod{p^{v_p(m) + l}}$ for every $m$ in $\{0,1,\dots,k\}$. This proves one direction of the statement.

Now, assume that we know $\tr{A^m} \pmod{p^{v_p(m)+l}}$ for all $m$ in $\{0,1,\dots,k\}$. By Theorem~\ref{thm:deep_power_sum_identity}, we have $b_1 = \tr{A}$, so $b_1 \pmod{p^l}$ is known. Now, assume by induction that we have determined $b_1, \dots, b_t \pmod{p^l}$ for some $1 \leq t < k$. By Theorem~\ref{thm:deep_power_sum_identity} with $m = t + 1$, we get an expression for $\tr{A^{t+1}}$ involving $b_1, \dots, b_{t+1}$. Note that the only term in this expression involving $b_{t+1}$ is $(-1)^{t+1}(t+1)b_{t+1}$. As we know $b_1, \dots, b_t \pmod{p^l}$, we can, by Corollary~\ref{cor:prime_power_we_know}, determine all the other terms in this expression modulo $p^{v_p(t+1)+l}$. As $\tr{A^{t+1}} \pmod{p^{v_p(t+1)+l}}$ is known by assumption, we can therefore compute $(t+1)b_{t+1} \pmod{p^{v_p(t+1)+l}}$, which allows us to determine $b_{t+1} \pmod{p^l}$. This completes the proof of the statement.
\end{proof}

It would be interesting to explore whether a more direct proof of Theorem~\ref{thm:main_theorem_6} exists, possibly using Sachs' Coefficients Theorem~\cite[Theorem 3.7]{wang2017simple}.

%%%%%%%%%%%%%%%%%%%%%%%%%%%%%%%%%%%%%%%%%%%%%%%%%%%%%%%%%%%%%%%%%%%%%%%%%%%%%%%%

\section{Reconstruction modulo 4}\label{sec:reconstruction_mod4}

In this section, we prove Theorems~\ref{thm:main_theorem_4} and~\ref{thm:main_theorem_5}. Before we proceed with the proofs we gather some other preliminary results that we will need. 

It is a result by Wang~\cite{wang2017simple} that the rank of the walk matrix $W^G$ over the finite field $\Fds_2$ is at most $\lceil \frac{n}{2}\rceil$. We will need the following lemma by Wang~\cite{wang2017simple} (or see Lemma 2.10 in~\cite{qiu2023smith}) which provides a more refined version of this result, presenting an explicit linear relation between the first columns of $W^G$ in terms of the coefficients of $\phi^G$. In what follows in this section, we write $\phi^G(x) = x^n - b_1 x^{n-1} + b_2 x^{n-2} - b_3 x^{n-3} + \cdots + (-1)^n b_n$.

\begin{lemma}[Lemma 3.9 in~\cite{wang2017simple}]\label{lem:wang_walk_matrix2} For every graph $G$,
\[
A^{\frac{n}{2}}\mathbf{1} + b_2 A^{\frac{n-2}{2}}\mathbf{1} + \cdots +b_{n-2} A\mathbf{1} +b_n\mathbf{1}\equiv 0\pmod{2},
\]
\noindent if $n$ is even, and,
\[
A^{\frac{n+1}{2}}\mathbf{1} + b_2 A^{\frac{n-1}{2}}\mathbf{1} + \cdots +b_{n-3} A^2\mathbf{1} +b_{n-1}A\mathbf{1}\equiv 0\pmod{2},
\]
\noindent if $n$ is odd.
\end{lemma}

Some additional results we need are the following propositions, which are the analogues of Proposition~\ref{prop:gen_basics} modulo $4$.

\begin{proposition}\label{prop:gen_basics_mod4} Consider $S(x) = \sum_{k \geq 0} \frac{s_k}{x^{k+1}}$ and $T(x) = \sum_{k \geq 0} \frac{t_k}{x^{k+1}}$, where $s_k, t_k \in \Zds_4$ for all $k \geq 0$. If we know the first $m$ coefficients of $S$ and $T$, we can  compute the first $m$ and $m+1$ coefficients of $S \pm T$ and $S \cdot T$, respectively. Furthermore, if we know the first $m+1$ and $m$ coefficients of $S \cdot T$ and $S$, respectively, and if the first coefficient of $S$ is invertible in $\Zds_4$, then we can  compute the first $m$ coefficients of $T$.
\end{proposition}
\begin{proof} We will prove only the last implication, as the others are analogous to those in Proposition~\ref{prop:gen_basics}. Let $(S \cdot T)(x) = \sum_{k \geq 0} \frac{u_k}{x^{k+1}}$, where $u_k$ is in $\Zds_4$ for all $k \geq 0$. Notice that $u_0 \equiv 0 \pmod{4}$ and $u_k \equiv \sum_{i=0}^{k-1} s_i t_{k-1-i}$ for $k \geq 1$, which can be rewritten as $s_0 t_{k-1} \equiv u_k - \sum_{i=1}^{k-1} s_i t_{k-1-i}$. Therefore, if we know $u_k$, $s_0, \dots, s_{k-1}$, and $t_0, \dots, t_{k-2}$, then, since $s_0$ is invertible in $\Zds_4$, we can solve for $t_{k-1}$. By iterating this procedure, we can recursively compute the first $m$ coefficients of $T$.
\end{proof}

\begin{proposition}\label{prop:gen_basics_mod4_2} Consider $S(x) = \sum_{k \geq 0} \frac{s_k}{x^{k+1}}$ with $s_k$ in $\Zds$ for every $k \geq 0$. If we know the first $m$ coefficients of $S \pmod{2}$, then we can  compute the first $m+1$ coefficients of $S^2 \pmod{4}$. If we know the first $m+1$ coefficients of $S^2 \pmod{4}$ and that the first coefficient of $S \pmod{2}$ is $1$, then we can  compute the first $m$ coefficients of $S \pmod{2}$.
\end{proposition}
\begin{proof} Let $S(x)^2=\sum_{k\geq 0}\frac{t_k}{x^{k+1}}$ with $t_k$ in $\Zds$ for every $k\geq 0$. Observe that $t_0=0$ and $t_k = \sum_{i=0}^{k-1}s_i s_{k-1-i}$ for every $k\geq 1$. Note that, $t_k = s_{(k-1)/2}^2+2\sum_{i=0}^{(k-3)/2}s_i s_{k-1-i}$ if $k\geq 1$ is odd, and $t_k = 2\sum_{i=0}^{(k-2)/2}s_i s_{k-1-i}$ if $k\geq 2$ is even. Since $z^2\pmod{4}$ and $2z\pmod{4}$ only depend on $z\pmod{2}$, if we know the first $m$ coefficients of $S \pmod{2}$, we can use the above expressions to compute the first $m+1$ coefficients of $S^2 \pmod{4}$.

Now assume that we know the first $m+1$ coefficients of $S^2 \pmod{4}$. Observe that we can write $2s_0 s_{k-1}=t_k-s_{(k-1)/2}^2-2\sum_{i=1}^{(k-3)/2}s_i s_{k-1-i}$ for $k \geq 1$ odd, and $2s_0 s_{k-1}=t_k-2\sum_{i=1}^{(k-2)/2}s_i s_{k-1-i}$ for $k\geq 2$ even. Thus, if we know $t_k\pmod{4}$ and $s_0, \dots, s_{k-2}\pmod{2}$, we can compute $2s_0 s_{k-1}\pmod{4}$. Since $s_0\equiv 1\pmod{2}$, we can then determine $s_{k-1}\pmod{2}$. By iterating this procedure, we can recursively compute the first $m$ coefficients of $S\pmod{2}$.
\end{proof}

Finally, observe that Proposition~\ref{prop:gen_basics_2} holds without modification if we assume that the coefficients of the polynomials and series are in $\Zds_4$ (or $\Zds_2$).

We are now ready to begin the proof of Theorem~\ref{thm:main_theorem_4}. To simplify the proof, we break it down into several steps. The following result is a direct consequence of Sachs' Coefficients Theorem~\cite[Theorem 3.7]{wang2017simple} and shows that the constant coefficient of $\phi^G$ is always even when $n$ is odd.

\begin{lemma}\label{lem:cte_mod2_even} For every graph $G$, we have $b_1=0$, and $b_k$ is even for every odd $k$. In particular, if $n$ is odd, the constant coefficient of $\phi^G$ is even.
\end{lemma}

Therefore, we can focus our reconstruction of the constant coefficient of $\phi^G \pmod{2}$ on the case where $n$ is even. We analyze separately the cases $n \equiv 0 \pmod{4}$ and $n \equiv 2 \pmod{4}$. We start with the following result.

\begin{proposition}\label{prop:first_coeficients_mod_4} Let $G$ be a graph with $n\geq 3$ vertices. Then, we can determine $w^{G \setminus i}(x) \pmod{4}$, the first $n$ coefficients of $w^G_{i,i}(x)$ and the first $\lceil\frac{n}{2}\rceil$ coefficients of $w^G(x)\pmod{4}$ and $w^G_i(x)\pmod{2}$ for every $i$ in $[n]$ from the polynomial deck of $G$.
\end{proposition}
\begin{proof} Assume that we know the polynomial deck $\{\phi^{G \setminus i}\}_{i \in [n]}$. By Theorem~\ref{thm:wang_series}, we can  compute $\{(\phi^{G \setminus i}, \phi^{\overline{G} \setminus i} \pmod{4})\}_{i \in [n]}$. Therefore, using Equation~\eqref{eq:walk} and Proposition~\ref{prop:gen_basics_2}, we can also compute $w^{G \setminus i}(x) \pmod{4}$ for every $i$ in $[n]$. By Theorem~\ref{thm:charac_derivative}, we can obtain all the coefficients of $\phi^G$ except the constant term, i.e., the top $n$ coefficients. Next, by applying Equation~\eqref{eq:closed_walk} and Proposition~\ref{prop:gen_basics_2}, we can compute the first $n$ coefficients of $w^G_{i,i}(x)$ for every $i$ in $[n]$. As a consequence, we also obtain the first $n$ coefficients of $\sum_{i \in [n]} w^G_{i,i}$, which gives us $\tr{A^k}$ for $k$ in $\{0, \dots, n-1\}$. Thus by using Equation~\eqref{eq:all_walks} and Theorem~\ref{thm:wang_series}, we can compute the first $\lceil \frac{n}{2} \rceil$ coefficients of $w^G(x) \pmod{4}$.

By Theorem~\ref{thm:godsil_mckay}, we have the equation
\begin{equation}\label{eq:godsil_mckay_mod4}
w_i^G(x)^2=(w^G(x)-w^{G\setminus i}(x))w_{i,i}^G(x)\pmod{4},
\end{equation}
\noindent and by Proposition~\ref{prop:gen_basics_mod4}, since we know the first $\lceil\frac{n}{2}\rceil$ coefficients of $w^G(x)\pmod{4}$, all coefficients of $w^{G \setminus i}(x)\pmod{4}$, and the first $n$ coefficients of $w_{i,i}^G(x)$, it follows that we can compute the first $\lceil\frac{n}{2} \rceil + 1$ coefficients of $w_i^G(x)^2 \pmod{4}$. By Proposition~\ref{prop:gen_basics_mod4_2}, as the first coefficient of $w^G_i$ is $1$, we can compute the first $\lceil\frac{n}{2}\rceil$ coefficients of $w_i^G(x)\pmod{2}$.
\end{proof}

The next result is one of the main mechanisms that makes the proof of Theorems~\ref{thm:main_theorem_4} and~\ref{thm:main_theorem_5} work and should be compared to Equation~\eqref{eq:main_mechanism} in the proof of Theorem~\ref{thm:main_theorem_2}.

\begin{lemma}\label{lem:main_mechanism_mod4} If we know $e_i^T A^k \mathbf{1}\pmod{2}$ for every $i$ in $[n]$, then we can compute $\mathbf{1}^T A^{2k}\mathbf{1}\pmod{4}$. 
\end{lemma}
\begin{proof} Note that,
\[
\mathbf{1}^TA^{2k}\mathbf{1}=
\mathbf{1}^TA^k \left(\displaystyle\sum_{i\in [n]}e_i e_i^T\right)A^k\mathbf{1}=
\displaystyle\sum_{i\in [n]} \left(e_i^TA^k\mathbf{1}\right)^2,
\]
\noindent for every $k$ in $\Zds_{\geq 0}$. Thus, since $z^2\pmod{4}$ depends only on $z\pmod{2}$, if we know $e_i^TA^k\mathbf{1} \pmod{2}$ for every $i$ in $[n]$, then we can compute $\mathbf{1}^TA^k\mathbf{1}\pmod{4}$.
\end{proof}

Now, we are ready to obtain the constant coefficient of $\phi^G \pmod{2}$ when $n \equiv 0 \pmod{4}$.

\begin{proposition}\label{prop:obtain_cte_mod2_0mod4} Let $G$ be a graph with $n \geq 4$ vertices, where $n \equiv 0 \pmod{4}$. Then, the constant coefficient of $\phi^G \pmod{2}$ can be computed from the polynomial deck of $G$.
\end{proposition}
\begin{proof} First note all the information we already have available due to Proposition~\ref{prop:first_coeficients_mod_4}. As we know the first $\frac{n}{2}$ coefficients of $w^G_i\pmod{2}$, in particular, by Equation~\eqref{eq:walk}, we know $e_i^T A^{\frac{n}{4}}\mathbf{1}\pmod{2}$ for every $i$ in $[n]$. Therefore, by Lemma~\ref{lem:main_mechanism_mod4}, we can determine $\mathbf{1}^T A^{\frac{n}{2}}\mathbf{1}\pmod{4}$. By Equation~\eqref{eq:all_walks}, it follows that we know the first $\frac{n}{2}+1$ coefficients of $w^G(x)\pmod{4}$. Therefore, by Equation~\eqref{eq:godsil_mckay_mod4} and Propositions~\ref{prop:gen_basics_mod4} and~\ref{prop:gen_basics_mod4_2}, since we know the first $\frac{n}{2}+1$ coefficients of $w^G(x)\pmod{4}$, all coefficients of $w^{G \setminus i}(x)\pmod{4}$, and the first $n$ coefficients of $w_{i,i}^G(x)$ and $n\geq 4$, we can compute the first $\frac{n}{2} + 1$ coefficients of $w_i^G(x) \pmod{2}$. Thus we know $e_i^T A^k \mathbf{1}\pmod{2}$ for every $i$ in $[n]$ and $k$ in $\{0,\dots, \frac{n}{2}\}$, or in other words, we know $A^k\mathbf{1}\pmod{2}$ for $k$ in this same set. By Lemma~\ref{lem:wang_walk_matrix2}, as $n$ is even and only the constant coefficient of $\phi^G$ is not known, we can then determine the constant coefficient of $\phi^G\pmod{2}$.
\end{proof}

Now, we focus on the case where $n \equiv 2 \pmod{4}$.

\begin{proposition}\label{prop:2mod4_coefficients} Let $G$ be a graph with $n \geq 6$ vertices, where $n \equiv 2 \pmod{4}$. Then, if $b_k$ is odd for some $k$ in $\{2,6,10,\dots, n-4\}$, we can determine $b_n\pmod{2}$.
\end{proposition}
\begin{proof} First, note that by Theorem~\ref{thm:charac_derivative}, we know all the coefficients $-b_1$, $\dots$, $(-1)^n b_{n-1}$ of $\phi^G$, except for $b_n$. We also have all the information already available from Proposition~\ref{prop:first_coeficients_mod_4}.

Since we know the first $\frac{n}{2}$ coefficients of $w^G_i \pmod{2}$, Equation~\eqref{eq:walk} allows us to determine $e_i^T A^k \mathbf{1}$ for every $i$ in $[n]$ and $k$ in $\left[\frac{n}{2} - 1\right]$. By Lemma~\ref{lem:main_mechanism_mod4}, as $n\geq 6$, we can then compute $\mathbf{1}^T A^k \mathbf{1} \pmod{4}$ for every $k$ in $\{\frac{n}{2} + 1, \frac{n}{2} + 3, \dots, n - 2\} $. By Proposition~\ref{thm:wang_series}, we have for every $k\geq 1$,  
\begin{equation}\label{eq:walks_reformulation}
\mathbf{1}^T A^k \mathbf{1} \equiv \frac{\tr{A^{2k}}}{2^{v_2(k)}} + \sum_{l=0}^{v_2(k)} \frac{\tr{A^{k/2^l}}}{2^{v_2(k)-l}} \pmod{4}.
\end{equation}

Furthermore, by Proposition~\ref{prop:first_coeficients_mod_4} we know the first $n$ coefficients of $\sum_{i \in [n]} w^G_{i,i}$, that is we know $\tr{A^k}$ for $k$ in $\{0,1,\dots, n-1\}$. Therefore, from the equation above, we can determine $\tr{A^m} \pmod{2^{v_2(m)+1}}$ for all $m$ in the set $S\defeq\{n+2,n+6,\dots, 2n-4\}$. By Theorem~\ref{thm:deep_power_sum_identity}, we have a formula for $\tr{A^m}$ for every $m$ in $S$, in terms of $b_1, b_2, \dots, b_n$. Observe that we already know the contributions to this formula involving only $b_1,b_2, \dots, b_{n-1}$. Now, we focus on identifying the contribution involving the coefficient $b_n$.

For a fixed $m$ in $S$, observe that the only terms in the formula for $\tr{A^m}$ that involve $b_n$ correspond to partitions $\lambda \vdash m$ that include a part of size $n$, that is $r_n(\lambda) \geq 1$. Since $m \leq 2n-4$ for every $m$ in $S$, it follows that $r_n(\lambda) = 1$ for these partitions. Thus, the contribution involving the coefficient $b_n$ to the formula for $\tr{A^m}$ given in Theorem~\ref{thm:deep_power_sum_identity} is  
\begin{equation}\label{eq:bn_contribution}
\sum_{\substack{\lambda\vdash m \\ r_n(\lambda)=1}} \frac{(-1)^{k(\lambda)+m} m}{k(\lambda)} \binom{k(\lambda)}{r_1(\lambda), r_2(\lambda), \dots} \prod_{j \geq 1} b_j^{r_j(\lambda)}.
\end{equation}

As we know $\tr{A^m} \pmod{2^{v_2(m)+1}}$ and the contribution to the formula of $\tr{A^m}$ not involving $b_n$, we can determine from the polynomial deck the content of Equation~\eqref{eq:bn_contribution} modulo $2^{v_2(m)+1}$. Also note that as $r_n(\lambda)=1$ and $k-n\leq n-4$ we have in this case
\[
\dfrac{m}{k(\lambda)}\binom{k(\lambda)}{r_1(\lambda), r_2(\lambda), \dots} = m\binom{k(\lambda)-1}{r_1(\lambda), \dots, r_{n-4}(\lambda)}.
\]

Observe that every partition $\lambda \vdash m$ above with $r_n(\lambda)=1$ corresponds to a partition $\lambda'\vdash m-n$ with $k(\lambda')=k(\lambda)-1$ and $r_j(\lambda')=r_j(\lambda)$ for every $j\leq n-4$. From Equation~\eqref{eq:bn_contribution} and the observations above, we conclude that we know
\[
b_n\sum_{\lambda\vdash m-n} (-1)^{k(\lambda)} m \binom{k(\lambda)}{r_1(\lambda), \dots, r_{n-4}(\lambda)} \prod_{j=1}^{n-4} b_j^{r_j(\lambda)} \pmod{2^{v_2(m)+1}},
\]
\noindent for every $m$ in $S$. By Lemma~\ref{lem:cte_mod2_even}, we know that $b_1 = 0$ and that $b_k$ is even for every odd $k$. Therefore, we can further simplify this expression, obtaining for every $m$ in $S$,
\[
b_n\sum_{\substack{\lambda\vdash m-n \\ r_k(\lambda) = 0, \text{ if } 2 \nmid k}} (-1)^{k(\lambda)} \binom{k(\lambda)}{r_2(\lambda), \dots, r_{n-4}(\lambda)} \prod_{\substack{j=2 \\ 2\mid j}}^{n-4} b_j^{r_j(\lambda)} \pmod{2}.
\]

Now, let $m'$ be the smallest element of $S$ such that $b_{m'-n}$ is odd. In this case, we can further simplify the above expression for $m'$. Observe that, since $m' - n \equiv 2 \pmod{4}$, every partition of $m' - n$ must include either an odd part or a part congruent to $2 \pmod{4}$. Since the partitions under consideration do not contain odd parts, it follows that every partition of $m' - n$ in the formula above must include at least one part congruent to $2 \pmod{4}$. However, by assumption, $b_k$ is even for every part $k < m' - n$ that is congruent to $2 \pmod{4}$, so the only nontrivial contribution in the expression above for $m'$ comes from the trivial partition $(m' - n) \vdash m' - n$. As a consequence, this expression simplifies to $-b_n b_{m'-n} \pmod{2}$. Since $b_{m'-n}$ is odd by assumption, it follows that we can determine $b_n \pmod{2}$.
\end{proof}

Before proceeding, we state the following useful observation.

\begin{lemma}[Lemma 2.5~\cite{ji2024new}]\label{lem:walks_even}
For every $k \geq 1$, $\mathbf{1}^T A^k \mathbf{1}$ is even.
\end{lemma}

Now, we are ready to obtain the constant coefficient of $\phi^G \pmod{2}$ when $n \equiv 2 \pmod{4}$. The proof of the next result is similar to the proof of Proposition~\ref{prop:2mod4_coefficients}.

\begin{proposition}\label{prop:obtain_cte_mod2_2mod4} Let $G$ be a graph with $n \geq 6$ vertices, where $n \equiv 2 \pmod{4}$. Then, the constant coefficient of $\phi^G \pmod{2}$ can be computed from the polynomial deck of $G$. 
\end{proposition}
\begin{proof} First, note that by Theorem~\ref{thm:charac_derivative}, we know all the coefficients $-b_1$, $\dots$, $(-1)^n b_{n-1}$ of $\phi^G$, except for $b_n$. We also have all the information already available from Proposition~\ref{prop:first_coeficients_mod_4}. Moreover, by Lemma~\ref{lem:cte_mod2_even} and Proposition~\ref{prop:2mod4_coefficients}, we can assume that $b_k$ is even for every $k$ odd or in $\{2,6,10,\dots, n-4\}$.

By Lemma~\ref{lem:wang_walk_matrix2}, for every $i$ in $[n]$, we have  
\[
e_i^T A^{\frac{n}{2}+1} \mathbf{1} + b_2 e_i^T A^{\frac{n}{2}} \mathbf{1} + \cdots + b_n e_i^T A \mathbf{1} \equiv 0 \pmod{2}.
\]  
Since $b_2 \equiv 0 \pmod{2}$ and we already know $b_1, \dots, b_{n-1}$ as well as $e_i^T A^k \mathbf{1}$ for $k$ in $\{0,1, \dots, \frac{n}{2}-1\}$, it follows that we can determine  
\[
e_i^T A^{\frac{n}{2}+1} \mathbf{1} + b_n e_i^T A \mathbf{1} \pmod{2}.
\]

As a consequence, using the same procedure as in the proof of Lemma~\ref{lem:main_mechanism_mod4}, we can determine,
\begin{align*}
\sum_{i\in[n]}\left(e_i^T A^{\frac{n}{2}+1} \mathbf{1} + b_n e_i^T A \mathbf{1}\right)^2 &= \mathbf{1}^T A^{n+2}\mathbf{1}+2b_n\mathbf{1}^TA^{\frac{n}{2}+1}\mathbf{1}+b_n^2 \mathbf{1}^TA^2\mathbf{1}\\ &\equiv\mathbf{1}^T A^{n+2}\mathbf{1}+b_n^2 \mathbf{1}^TA^2\mathbf{1}\\ &\equiv\mathbf{1}^T A^{n+2}\mathbf{1}+b_n \mathbf{1}^TA^2\mathbf{1}\pmod{4},
\end{align*}
\noindent where, in the penultimate and last congruence, we use Lemma~\ref{lem:walks_even} and the fact that $b_n^2 \mathbf{1}^T A^2 \mathbf{1}\pmod{4}$ only depends on the parity of $b_n$. By Equation~\eqref{eq:walks_reformulation}, and since we already know $\tr{A^k}$ for $k$ in $\{0,1,\dots, n-1\}$, it follows that we can determine
\[
\tr{A^{2n+4}}+\tr{A^{n+2}}+b_n\left(\tr{A^4}+\tr{A^2}+2\tr{I}\right)2^{v_2(n+2)-1}\pmod{2^{v_2(n+2)+2}}.
\]

We now analyze each term in this equation. Note that, by Theorem~\ref{thm:deep_power_sum_identity},
\begin{align*}
\tr{A^4}+\tr{A^2}+2\tr{I} &=(2 b_2^2-4 b_4)+(-2 b_2)+2n\\&=2( b_2^2-2 b_4-b_2+n)\\&\equiv 2(-2 b_4-b_2+2)\pmod{8},
\end{align*}
\noindent where, in the last congruence, we used the fact that $b_2$ is even and $n\equiv 2\pmod{4}$. Note that in the formula for $\tr{A^{n+2}}$ given in Theorem~\ref{thm:deep_power_sum_identity}, as $b_1=0$ by Lemma~\ref{lem:cte_mod2_even}, the only term that contains $b_n$ corresponds to the partition $(n,2) \vdash n+2$ and is $-(n+2)b_2 b_n$. Since $b_1, \dots, b_{n-1}$ are already known, by the discussion above we know 
\[
\tr{A^{2n+4}} - (n+2) b_2 b_n + b_n(-2 b_4 - b_2 + 2) 2^{v_2(n+2)} \pmod{2^{v_2(n+2)+2}},
\]

As $b_2$ is even, $- (n+2) b_2 b_n-b_2 b_n2^{v_2(n+2)}\equiv 0\pmod{2^{v_2(n+2)+2}}$, and therefore, we know,
\begin{equation}\label{eq:last_trace_analysis}
\tr{A^{2n+4}} + b_n(-b_4+1) 2^{v_2(n+2)+1} \pmod{2^{v_2(n+2)+2}}.
\end{equation}

Now, we analyze the contribution of $b_n$ in the formula for $\tr{A^{2n+4}}$ given in Theorem~\ref{thm:deep_power_sum_identity}. This contribution arises from partitions $\lambda \vdash 2n+4$ where $r_n(\lambda)$ is either $1$ or $2$. 

First, consider the contribution of $b_n$ arising from partitions $\lambda \vdash 2n+4$ with $r_n(\lambda) = 1$. By an analysis analogous to the one presented in the proof of Proposition~\ref{prop:2mod4_coefficients}, the contribution of $b_n$ in this case is  
\[
b_n\sum_{\substack{\lambda\vdash n+4 \\ r_k(\lambda) = 0, \text{ if } 2 \nmid k\\ \text{ or } k=n}} (-1)^{k(\lambda)+1} (2n+4) \binom{k(\lambda)}{r_2(\lambda), \dots, r_{n-2}(\lambda)} \prod_{\substack{j=2 \\ 2\mid j}}^{n-2} b_j^{r_j(\lambda)}.
\]  
Since $n+4\equiv 2\pmod{4}$, every partition $\lambda \vdash n+4$ must contain either an odd part or a part congruent to $2\pmod{4}$. Given that $b_k$ is even for every $k$ in $\{2,6,10,\dots, n-4\}$, we conclude that the expression above is divisible by $2^{v_2(n+2)+2}$. Therefore, the contribution of $b_n$ to $\tr{A^{2n+4}}$ from partitions $\lambda \vdash 2n+4$ with $r_n(\lambda) = 1$ is $0\pmod{2^{v_2(n+2)+2}}$.

We now examine the contribution of $b_n$ to $\operatorname{tr}(A^{2n+4})$ from partitions $\lambda \vdash 2n+4$ where $r_n(\lambda) = 2$. Notice that every such partition $\lambda$ corresponds to a partition $\lambda' \vdash 4$ with $k(\lambda') = k(\lambda) - 2$. As a result, the contribution in this case is simply
\[
b_n^2\left((-1)^4\dfrac{2n+4}{3}\binom{3}{2,1} b_4+(-1)^5\dfrac{2n+4}{4}\binom{4}{2,2}b_2^2\right).
\]  
Since $b_2$ is even, we conclude that the contribution of $b_n$ to $\operatorname{tr}(A^{2n+4})$ from partitions $\lambda \vdash 2n+4$ with $r_n(\lambda) = 2$ is $(2n+4)b_n^2 b_4\pmod{2^{v_2(n+2)+2}}$. By Equation~\eqref{eq:last_trace_analysis}, we obtain that we know
\[
(2n+4)b_n^2 b_4 + b_n(-b_4+1) 2^{v_2(n+2)+1} \pmod{2^{v_2(n+2)+2}}.
\]
As a consequence, we can determine
\[
\dfrac{(2n+4)}{2^{v_2(2n+4)}}b_n^2 b_4 + b_n(-b_4+1) \pmod{2}.
\]
\noindent Since $b_n^2 \equiv b_n \pmod{2}$ and  
\[
\frac{(2n+4)}{2^{v_2(2n+4)}} b_n b_4 - b_n b_4 \equiv 0 \pmod{2},
\]
\noindent we conclude that we can determine $b_n \pmod{2}$.
\end{proof}

By combining Lemma~\ref{lem:cte_mod2_even} with Propositions~\ref{prop:obtain_cte_mod2_0mod4} and~\ref{prop:obtain_cte_mod2_2mod4}, we are now ready to prove Theorem~\ref{thm:main_theorem_4}.

\begin{proof}[Proof of Theorem~\ref{thm:main_theorem_4}] First, note that by Theorem~\ref{thm:charac_derivative}, we know the coefficients $-b_1$, $\dots$, $(-1)^n b_{n-1}$ of $\phi^G$, and by Lemma~\ref{lem:cte_mod2_even} and Propositions~\ref{prop:obtain_cte_mod2_0mod4} and~\ref{prop:obtain_cte_mod2_2mod4}, we also know $b_n\pmod{2}$. We also have all the information already available from Proposition~\ref{prop:first_coeficients_mod_4}.

Since we know the first $\lceil \frac{n}{2} \rceil$ coefficients of $w^G_i(x) \pmod{2}$ for every $i$ in $[n]$, we can determine the first $\lceil \frac{n}{2} \rceil$ columns of $W^G \pmod{2}$. By Lemma~\ref{lem:wang_walk_matrix2}, as we know $\phi^G \pmod{2}$ and $A^k \mathbf{1}$ for $k$ in $\{0, \dots, \lceil \frac{n}{2} \rceil-1\}$ we can compute $A^k \mathbf{1} \pmod{2}$ for every $k$ in $\Zds_{\geq 0}$. In particular, we can determine $W^G \pmod{2}$ and $w^G_i(x) \pmod{2}$ for every $i$ in $[n]$. Therefore, by Equation~\eqref{eq:godsil_mckay_mod4} and Propositions~\ref{prop:gen_basics_mod4} and~\ref{prop:gen_basics_mod4_2}, since we know $w^G_i(x)\pmod{2}$ and the first $n$ coefficients of $w_{i,i}^G(x)$, we can compute the first $n$ coefficients of $w^G(x) \pmod{4}$.

By Equation~\eqref{eq:all_walks} and Proposition~\ref{prop:gen_basics_2}, since we know the top $n$ coefficients of $\phi^G$ and the first $n$ coefficients of $w^G(x) \pmod{4}$, we can determine the top $n$ coefficients of $(-1)^n \phi^{\overline{G}}(-x-1) - \phi^G(x) \pmod{4}$. This allows us to compute the top $n$ coefficients of $\phi^{\overline{G}} \pmod{4}$.

By Lemma~\ref{lem:walks_even}, we know $w^G(x) \pmod{2}$. From Equation~\eqref{eq:all_walks} and Proposition~\ref{prop:gen_basics_2}, since we know $\phi^G \pmod{2}$, we can determine $(-1)^n \phi^{\overline{G}}(-x-1) - \phi^G(x) \pmod{2}$, and therefore, the constant coefficient of $\phi^{\overline{G}} \pmod{2}$. This finishes the proof.
\end{proof}

Let $\text{rank}_2(W^G)$ denote the rank of $W^G$ over $\mathbb{F}_2$. We observe that the proof of Proposition~\ref{prop:obtain_cte_mod2_2mod4}, and consequently Theorem~\ref{thm:main_theorem_4}, becomes simpler under the additional assumption that $\text{rank}_2(W^G) < \lceil \frac{n}{2} \rceil$. In this case, a strategy analogous to the one used in the proof of Theorem~\ref{thm:main_theorem_3} can be applied. Therefore, the main difficulty in proving Proposition~\ref{prop:obtain_cte_mod2_2mod4} lies in handling the case where $\text{rank}_2(W^G) = \lceil \frac{n}{2} \rceil$. As discussed in Section~\ref{sec:introduction}, and as we will soon see, this situation is analogous to the one in Theorem~\ref{thm:main_theorem_5}.

\begin{proof}[Proof of Theorem~\ref{thm:main_theorem_5}] First, note that all the information already available from Proposition~\ref{prop:first_coeficients_mod_4} and Theorem~\ref{thm:main_theorem_4}. In particular, note that as argued in the proof of Theorem~\ref{thm:main_theorem_4} we can compute the first $n$ coefficients of $w^G(x) \pmod{4}$. We claim that it is possible to determine the $(n+1)$--th coefficient of $w^G(x)\pmod{4}$, that is $\mathbf{1}^T A^n\mathbf{1}\pmod{4}$. 

If $n$ is even, then by Equation~\eqref{eq:walk}, we know $e_i^T A^{\frac{n}{2}} \mathbf{1} \pmod{2}$ for every $i$ in $[n]$. As a consequence, by Lemma~\ref{lem:main_mechanism_mod4}, we can determine $\mathbf{1}^T A^n \mathbf{1} \pmod{4}$, as desired. Therefore, assume that $n$ is odd and $\text{rank}_2(W^G) < \lceil \frac{n}{2} \rceil$. 

Note that, by Theorem~\ref{thm:main_theorem_4}, we have access to $W^G \pmod{2}$. Since $\text{rank}_2(W^G)$ is less than $\lceil \frac{n}{2} \rceil$, there is a linear relation between the columns $\{0, \dots, \frac{n-1}{2}\}$ of $W^G \pmod{2}$. By possibly multiplying this linear relation by $A$, we can assume that:
\[
v \coloneqq A^{\frac{n-1}{2}} \mathbf{1} + d_1 A^{\frac{n-3}{2}} \mathbf{1} + \cdots + d_{\frac{n-3}{2}} A \mathbf{1} + d_{\frac{n-1}{2}} \mathbf{1} \equiv 0 \pmod{2},
\]
for some $d_1, \dots, d_{\frac{n-1}{2}}$ in $\mathbb{F}_2$. Thus, we have $v^T A v \equiv 0 \pmod{4}$. Moreover, by Lemma~\ref{lem:walks_even}, $\mathbf{1}^T A^k \mathbf{1}$ is even for all $k \geq 1$, and we can simplify this expression to:  
\[
\mathbf{1}^T A^n \mathbf{1} + d_1^2 \mathbf{1}^T A^{n-2} \mathbf{1} + \cdots + d_{\frac{n-1}{2}}^2 \mathbf{1}^T A \mathbf{1} \equiv 0 \pmod{4}.
\]
\noindent From this final expression, we can now determine $\mathbf{1}^T A^n \mathbf{1} \pmod{4}$. This completes the proof of our claim.

By Equation~\eqref{eq:godsil_mckay_mod4} and Propositions~\ref{prop:gen_basics_mod4} and~\ref{prop:gen_basics_mod4_2}, since we know $w^G_i(x) \pmod{2}$ and the first $n+1$ coefficients of $w^G(x)$, we can compute the first $n+1$ coefficients of $w^G_{i,i}(x) \pmod{4}$. By Equation~\eqref{eq:closed_walk} and Proposition~\ref{prop:gen_basics_2}, since the first coefficient of $w^G_{i,i}(x) \pmod{4}$ is $1 \pmod{4}$, we can determine the constant coefficient of $\phi^G \pmod{4}$. Finally, by Equation~\eqref{eq:all_walks} and Proposition~\ref{prop:gen_basics_2}, as we know $\phi^G \pmod{4}$ and the first $n+1$ coefficients of $w^G(x) \pmod{4}$, we can determine $(-1)^n \phi^{\overline{G}}(-x-1) - \phi^G(x) \pmod{4}$, and therefore compute $\phi^{\overline{G}} \pmod{4}$.
\end{proof}

%%%%%%%%%%%%%%%%%%%%%%%%%%%%%%%%%%%%%%%%%%%%%%%%%%%%%%%

\section{Motivation and questions}\label{sec:questions}

We now focus on some of the motivation for this work, which comes from homomorphism counts. Let $Hom(F, G)$ denote the number of homomorphisms from the graph $F$ to the graph $G$. Define $Hom(\mathcal{F}, G)$ as the vector $(Hom(F, G))_{F \in \mathcal{F}}$, where $\mathcal{F}$ is a family of graphs. A well-known result by Lovász~\cite{lovasz1967operations} states that two graphs $G$ and $H$ are isomorphic if and only if $Hom(\mathcal{G}, G) = Hom(\mathcal{G}, H)$, where $\mathcal{G}$ is the family of all graphs. Other results in this context are also known~\cite{dell2018lovasz, grohe2022homomorphism}. Let $\mathcal{C}$, $\mathcal{P}$, and $\mathcal{T}$ represent the families of all cycles, paths, and trees, respectively. Then, we have the following:
\begin{itemize} 

\item $\phi^G = \phi^H$ if and only if $Hom(\mathcal{C}, G) = Hom(\mathcal{C}, H)$; 

\item $(\phi^G, \phi^{\overline{G}}) = (\phi^H, \phi^{\overline{H}})$ if and only if $Hom(\mathcal{C} \cup \mathcal{P}, G) = Hom(\mathcal{C} \cup \mathcal{P}, H)$; 

\item $G$ and $H$ are fractionally isomorphic if and only if $Hom(\mathcal{T}, G) = Hom(\mathcal{T}, H)$; 

\item $G$ and $H$ are isomorphic if and only if $Hom(\mathcal{G}, G) = Hom(\mathcal{G}, H)$. \end{itemize}

Given two families of graphs $\mathcal{F}$ and $\mathcal{H}$, we say that $Hom(\mathcal{F}, \cdot)$ is reconstructible from $Hom(\mathcal{H}, \cdot)$ if, for every graph $G$ with at least three vertices, $Hom(\mathcal{F}, G)$ is uniquely determined by $Hom(\mathcal{H}, G \setminus i)$ for every $i$ in $[n]$. If $Hom(\mathcal{F}, \cdot)$ is reconstructible from $Hom(\mathcal{F}, \cdot)$, we say that $Hom(\mathcal{F}, \cdot)$ is reconstructible.

In this context, the Ulam-Kelly reconstruction conjecture~\cite{kelly1942isometric} and polynomial reconstruction problem (Question~\ref{question:polynomial_reconstruction}) can be restated as:

\begin{conjecture}[Ulam-Kelly]\label{conj:reconstruction} $Hom(\mathcal{G}, \cdot)$ is reconstructible. \end{conjecture}

\begin{question}[Polynomial Reconstruction]\label{conj:polynomial_reconstruction} Is $Hom(\mathcal{C}, \cdot)$ reconstructible? \end{question}

Furthermore, Theorems~\ref{thm:hagos} and~\ref{thm:main_theorem} imply the following:

\begin{theorem} $Hom(\mathcal{C} \cup \mathcal{P}, \cdot)$ is reconstructible. \end{theorem}

The main motivation for this work was to understand why $Hom(\mathcal{C} \cup \mathcal{P}, \cdot)$ is reconstructible, in contrast to $Hom(\mathcal{C}, \cdot)$ or $Hom(\mathcal{G}, \cdot)$, for which reconstructibility is not known.

These results and problems also motivate the following questions, which were also raised by Ramana, Scheinerman, and Ullman~\cite{ramana1994fractional, scheinerman2011fractional}:

\begin{question}\label{question:fractional_reconstruction_1} Is $Hom(\mathcal{T}, \cdot)$ reconstructible from $Hom(\mathcal{G}, \cdot)$? \end{question}

\begin{question}\label{question:fractional_reconstruction_2} Is $Hom(\mathcal{T}, \cdot)$ reconstructible? \end{question}

We note that a positive solution to either of these questions implies a positive solution to a long-standing question by Kocay~\cite[Conj. 5.1]{kocay2006some}, which asks whether the number of spanning trees by isomorphism type of a graph can be reconstructed from its deck. While Questions~\ref{question:fractional_reconstruction_1} and~\ref{question:fractional_reconstruction_2} might seem approachable using the techniques presented in Section~\ref{sec:proof_main_theorem}, along with recent results in~\cite{grohe2022homomorphism}, a recent work by Gonçalves and Thatte~\cite{gonccalves2023refinement} shows that when considering homomorphism counts of trees it is necessary to deal with the obstacle of pseudo-similar vertices.

%%%%%%%%%%%%%%%%%%%%%%%%%%%%%%%%%%%%%%%%%%%%%%%%%%%%%%%%%%%%%%%%%%%%%%%%%%%%%%%%

\section*{Acknowledgements}

We acknowledge fruitful conversations about the topic of this paper with Gabriel Coutinho and Chris Godsil.

%%%%%%%%%%%%%%%%%%%%%%%%%%%%%%%%%%%%%%%%%%%%%%%%%%%%%%%%%%%%%%%%%%%%%%%%%%%%%%%%
\bibliographystyle{plain}
\IfFileExists{references.bib}
{\bibliography{references.bib}}
{\bibliography{../references}}

%%%%%%%%%%%%%%%%%%%%%%%%%%%%%%%%%%%%%%%%%%%%%%%%%%%%%%%%%%%%%%%%%%%%%%%%%%%%%%%%
	
\end{document}